\documentclass[12pt]{article}

\usepackage{amsmath,amssymb,amsthm}
\usepackage{latexsym}
\usepackage{graphicx}
\usepackage{psfrag}

\newtheorem{theorem}{Theorem}
\newtheorem{lemma}{Lemma}

\def\nfrac#1#2{{\textstyle\frac{#1}{#2}}}
\def\dfrac#1#2{\lower0.15ex\hbox{\large$\frac{#1}{#2}$}}

\begin{document}

\title{Corrigendum: Sampling regular graphs and a peer-to-peer network}

\author{Colin Cooper\and Martin Dyer \and Catherine Greenhill}
\author{
Colin Cooper \\
\small Department of Computer Science\\[-0.5ex]
\small Kings College \\[-0.5ex]
\small London WC2R 2LS, U.K.\\[-0.5ex]
\small\tt ccooper@dcs.kcl.ac.uk\\
\and
Martin Dyer\\
\small School of Computing\\[-0.5ex]
\small University of Leeds\\[-0.5ex]
\small Leeds LS2 9JT, U.K.\\[-0.5ex]
\small \tt dyer@comp.leeds.ac.uk\\
\and
Catherine Greenhill\\
\small School of Mathematics and Statistics\\[-0.5ex]
\small The University of New South Wales\\[-0.5ex]
\small Sydney NSW 2052, Australia\\[-0.5ex]
\small \tt csg@unsw.edu.au}

\date{21 February 2012}

\maketitle

\begin{abstract}
In [\emph{Combinatorics, Probability and Computing} {\bf 16} (2007), 
557--593, Theorem 1] we proved a polynomial-time bound on
the mixing rate of the switch chain for sampling $d$-regular graphs.  
This corrigendum corrects a technical error in the proof.  In order to
fix the error, we must multiply the bound on the mixing time by
a factor of $d^8$. 
\end{abstract}

For notation and terminology not defined here, see~\cite{CDG}.
Let $\Omega_{n,d}$ be the set of all $d$-regular graphs on the
vertex set $\{ 1,2,\ldots, n\}$.  A natural Markov chain
for sampling elements of $\Omega_{n,d}$ is studied in~\cite{CDG}.
We will refer to this chain as the \emph{switch chain}.
The transition procedure is given 
in~\cite[Figure 1]{CDG}: from a current state in $\Omega_{n,d}$ it chooses
two non-incident edges uniformly at random, and with probability
$\nfrac{1}{2}$ it chooses a perfect matching of the four endvertices
uniformly at random, replacing the two chosen edges with this
perfect matching unless multiple edges would result.  (With the
remaining probability $\nfrac{1}{2}$ it does nothing, as this
is a lazy chain.)  Our main theorem stated that the mixing time of
this chain is bounded above by
\begin{equation}
\label{old-tau}
 \tau(\varepsilon)\leq d^{15} \,n^8 \, (dn\log(dn) + \log(\varepsilon^{-1})).
\end{equation}
There was a slight error in the proof of one of the
lemmas used to establish this result, namely~\cite[Lemma 5]{CDG}.
Correcting the lemma leads to the new upper bound given below,
which is a factor of $d^8$ larger.  (It is possible that with a more sophisticated
argument a smaller power of $d$ might be enough.  But we choose to use
a simple argument here.)

\begin{theorem}
Let $\tau(\varepsilon)$ be the mixing time of the switch Markov chain
with state space $\Omega_{n,d}$.  Then
\label{new-tau}
\[  \tau(\varepsilon)\leq d^{23}\, n^8 \,(dn\log(dn) + \log(\varepsilon^{-1})).
\]
\label{main}
\end{theorem}

The error arose in~\cite[Lemma 5]{CDG}, which bounds
the flow $f(e)$ carried by an arbitrary transition $e$ of the switch
Markov chain.  We give a corrected statement of this lemma
below
and explain the changes required to correct the proof.

\begin{lemma}
For any transition $e=(Z,Z')$ of the Markov chain,
\[
f(e)\leq 2d^{20}n^5/|\Omega_{n,d}|.
\]
\label{right-f}
\end{lemma}

\begin{proof}
Fix a transition $e=(Z,Z')$ of the switch chain, and let
$(G,G')\in\Omega_{n,d}\times\Omega_{n,d}$ be such that
$e$ lies on the canonical path $\gamma_\psi(G,G')$ from $G$ to $G'$
corresponding to the pairing $\psi\in\Psi(G,G')$.    From $e$ and
$(G,G',\psi)$, the encoding $L$ and a yellow-green colouring of
$H = G\triangle G' = Z\triangle L$ can be formed, with green
edges belonging to $Z$ and yellow edges belong to $L$.
As in the
proof of~\cite[Lemma 5]{CDG}, the pairing $\psi$ produces a circuit
decomposition $\mathcal{C}$ of $H$ with colours alternating
yellow, green almost everywhere.

Call a vertex $x$ \emph{bad} with respect to $\psi$ if two edges
of the same colour are paired at $x$ under $\psi$.
If a vertex is not bad, it is called \emph{good}.  Every bad
vertex lies on the circuit currently being processed.
Specifically, bad vertices may only be found next to 
\emph{interesting edges}, which are
\begin{itemize}
\item odd chords which have been switched during the processing of
a 1-circuit (but not switched back), or
\item the shortcut edge, switched while processing a 2-circuit
(but not yet switched back).
\end{itemize}
Recall that an edge is \emph{bad} if it receives label $-1$ or 2
in the encoding $L$.
Note that all bad edges are interesting (an interesting edge is bad
if and only if it does not lie in the symmetric difference $H$).
Therefore the interesting edges form a subgraph of one of the
graphs shown in~\cite[Figure 10]{CDG}, and~\cite[Lemma 2]{CDG}
also describes interesting edges: there are at most four interesting
edges and hence at most six bad vertices.

A yellow-yellow or green-green pair at a bad vertex $x$ is
called a \emph{bad pair}.  In the proof of~\cite[Lemma 5]{CDG}
we proved that two quantities $|\Psi'(L)|$ and $|\Psi(H)|$
are related by the inequality
\begin{equation}
\label{relation}
 |\Psi'(L)|\leq d^b \, |\Psi(H)|
\end{equation}
where $b$ denotes the maximum number of bad pairs in any encoding
$L$.  To prove this we showed that each bad pair contributes
a factor of at most $d$ to the right hand side of this inequality.
In~\cite{CDG} this equation appeared as~\cite[Equation (6)]{CDG}
with $b=6$, as 
we claimed that there are at most 6 bad pairs in any encoding.
As we will see, the true maximum value is $d=14$.
We explain this below, and then extend the argument from~\cite{CDG} to prove that each 
bad pair contributes a factor of at most $d$
to the right hand side of (\ref{relation}).  
Hence~\cite[Equation (6)]{CDG} can be replaced by
\begin{equation}
\label{relation14}
 |\Psi'(L)|\leq d^{14} \, |\Psi(H)|,
\end{equation}
giving an extra factor of $d^8$.

First we explain why there can be up to 14 bad pairs in an encoding.
It is true that, in the circuit $C$ currently being processed,
a bad pair will occur at each bad vertex which is the endvertex
of an interesting edge.  This accounts for up to six bad pairs.
However, in~\cite{CDG} we overlooked the fact that the
interesting edges may themselves belong to the symmetric difference
$H$.   For each such interesting edge there will be a bad pair
at each endvertex of the edge, in the circuit containing the edge.
As there are at most four interesting edges, this gives \emph{up to
eight additional bad pairs}, giving a total of at most 14 bad pairs
in any encoding.   (An example is given below which shows that
the upper bound of 14 can be achieved.)

Next we explain why (\ref{relation}) still holds, now with $b=14$.
In the proof of~\cite[Equation (1)]{CDG}  we stated that a
bad vertex was incident with at most one bad pair of each colour.
This limited the cases that we considered when calculating the
number of ways to pair up the edges around each bad vertex.
Now when the eight additional bad pairs are taken into
consideration, we see that it is possible to have up to two bad
pairs of each colour present at a bad vertex.  But these extra
cases can be dealt with using very similar arguments to those
in~\cite{CDG}, showing that every bad pair contributes a factor
of at most 
$d$ to the number of pairings at that vertex.  For example,
suppose that a vertex $v$ is incident with $\theta_v + 2$ green arcs and
$\theta_v - 2$ yellow arcs. Then $v$ is bad, with two bad green 
pairs and no bad yellow pairs.  The number of ways to pair up the
edges around $v$ is then
\[ 3 \binom{\theta_v + 2}{4}\, (\theta_v - 2)!
  = \frac{(\theta_v + 2)(\theta_v + 1)}{8}\, \theta_v!
   \leq \theta_v^2\, \theta_v! \leq d^2\, \theta_v!.
\]
This extra factor of $d^2$ arises from the two bad pairs at $v$,
as required.  All other cases can be dealt with similarly.
This shows that (\ref{relation14}) holds.
The extra factor of $d^8$ carries through the rest of the proof
of~\cite[Lemma 5]{CDG},
leading to the corrected bound stated here.
\end{proof}

\begin{proof}[Proof of Theorem~\ref{main}]
The additional factor of $d^8$ from Lemma~\ref{right-f}
is carried throughout the
rest of the proof of Theorem~\ref{main}.   For example,
the bound on the load of the flow, $\rho(f)$, given in~\cite[Equation (7)]{CDG}
must become
\[ \rho(f)\leq 2 d^{22} \,n^7.\]
Finally, the extra factor $d^8$ appears in the bound on the mixing time,
giving us (\ref{new-tau}).
\end{proof}

We now work through an example which shows that there can indeed
be 14 bad pairs in the yellow-green colouring of $H$.
(See~\cite{CDG} for a detailed description of the procedures used.)

The symmetric difference $H$ that we use is shown in Figure~\ref{example1}.
Initially, solid lines belong to $G$ and dashed lines belong to $G'$.
One additional edge, $x_0 x_5$, is shown using a dotted arc: this
edge belongs to $G\cap G'$ but will be used to create the canonical
path, so we include it in our figures.   
\begin{figure}[ht]
 \psfrag{v}{$x_0$}\psfrag{x00}{$a$}\psfrag{y00}{$x_5$}
 \psfrag{y10}{$x_2$}\psfrag{x10}{$x_1$}\psfrag{x11}{$x_{11}$}
 \psfrag{y11}{$x_{10}$}\psfrag{y01}{$x_7$}\psfrag{x01}{$x_6$}
 \psfrag{u1}{$x_4$} \psfrag{u2}{$x_3$} \psfrag{u3}{$x_9$} \psfrag{u4}{$x_8$}
 \psfrag{p1}{$z_1$} \psfrag{p2}{$z_2$} \psfrag{w1}{$w_1$} \psfrag{w2}{$w_2$}
 \psfrag{r1}{$u_1$} \psfrag{r2}{$u_2$} \psfrag{t1}{$v_1$} \psfrag{t2}{$v_2$}
\centerline{\includegraphics[scale=0.6]{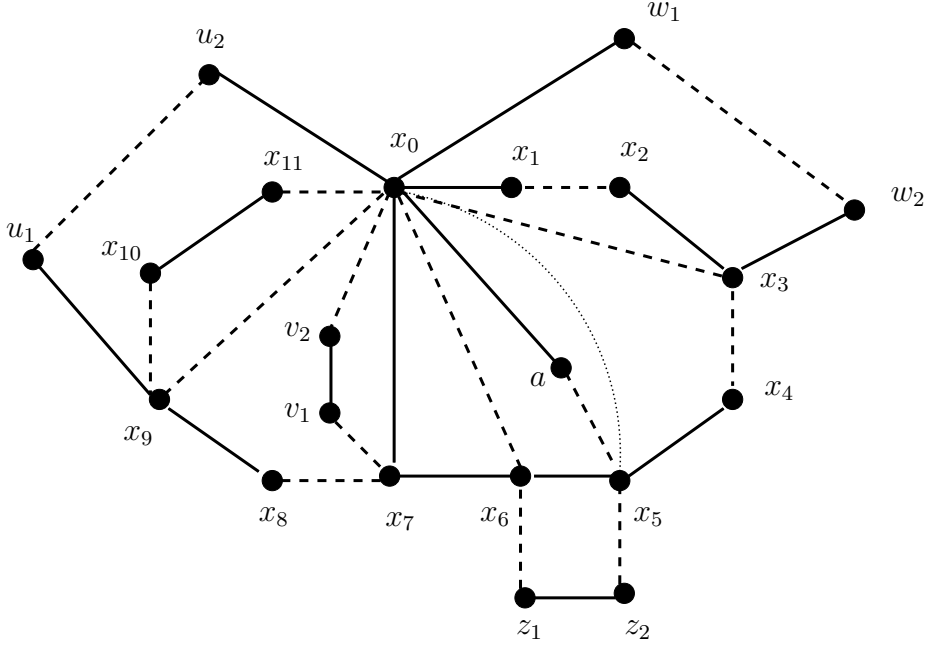}}
\caption{The symmetric difference $H$ (solid for $Z$ and
dashed for $H - Z$) together with one additional edge (dotted).
Initially $Z=G$ and $H-Z = G'$.}
\label{example1}
\end{figure}

Suppose that we work with the pairing $\psi$ that produces the following 
sequence of circuits $\mathcal{C}$:
\[
 x_0 a x_5 x_4 x_3 x_2 x_1  x_0 x_{11} x_{10} x_9 x_8  x_7 x_6,
 \,\,\,
       x_0 x_9 u_1 u_2,\,\,\, x_0 w_1 w_2 x_3,\,\,\, x_0 x_7 v_1 v_2,\,\,\,
            x_5 x_6 z_1 z_2.
            \]
These circuits are processed in the given order, so the first circuit to be
processed is the longest one.
It is a 2-circuit in case (b) with shortcut edge $x_5 x_6$. 
(To match up the notation with~\cite[Figure 3]{CDG}, use the
transformation $(x_0, a, x_5, x_6) \mapsto (v, w, x, y)$:
our choice of notation will prove convenient when processing
the 1-circuit arising from this 2-circuit.)
After performing the shortcut switch we obtain the situation of Figure~\ref{example2}.
Interesting edges will be shown very thickly, and bad vertices will also
be drawn larger.  Bad pairs are indicated by thin arcs joining the two edges
in the pair. 
\begin{figure}[ht]
 \psfrag{v}{$x_0$}\psfrag{x00}{$a$}\psfrag{y00}{$x_5$}
 \psfrag{y10}{$x_2$}\psfrag{x10}{$x_1$}\psfrag{x11}{$x_{11}$}
 \psfrag{y11}{$x_{10}$}\psfrag{y01}{$x_7$}\psfrag{x01}{$x_6$}
 \psfrag{u1}{$x_4$} \psfrag{u2}{$x_3$} \psfrag{u3}{$x_9$} \psfrag{u4}{$x_8$}
 \psfrag{p1}{$z_1$} \psfrag{p2}{$z_2$} \psfrag{w1}{$w_1$} \psfrag{w2}{$w_2$}
 \psfrag{r1}{$u_1$} \psfrag{r2}{$u_2$} \psfrag{t1}{$v_1$} \psfrag{t2}{$v_2$}
\centerline{\includegraphics[scale=0.6]{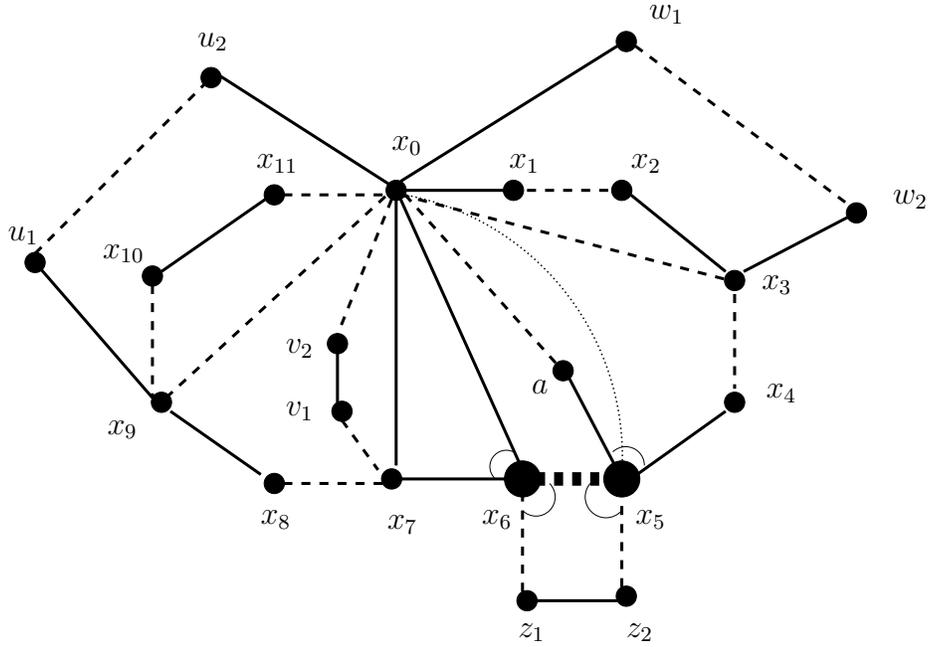}}
\caption{After the shortcut switch, with 1 interesting edge, 2 bad vertices, 4 bad pairs.}
\label{example2}
\end{figure}
For example, vertices $x_6$ and $x_5$ are bad vertices
in Figure~\ref{example2} with one bad pair of each colour at each vertex.
The shortcut edge $\{ x_6,x_5\}$ is interesting.

Next we must process the 1-circuit containing the shortcut edge,
which is 
\[ x_0 x_1 x_2 x_3 x_4 x_5 x_6 x_7 x_8 x_9 x_{10} x_{11}.\]  
There will be three
phases, and the first phase only takes one step, producing Figure~\ref{example3}.
\begin{figure}[ht]
 \psfrag{v}{$x_0$}\psfrag{x00}{$a$}\psfrag{y00}{$x_5$}
 \psfrag{y10}{$x_2$}\psfrag{x10}{$x_1$}\psfrag{x11}{$x_{11}$}
 \psfrag{y11}{$x_{10}$}\psfrag{y01}{$x_7$}\psfrag{x01}{$x_6$}
 \psfrag{u1}{$x_4$} \psfrag{u2}{$x_3$} \psfrag{u3}{$x_9$} \psfrag{u4}{$x_8$}
 \psfrag{p1}{$z_1$} \psfrag{p2}{$z_2$} \psfrag{w1}{$w_1$} \psfrag{w2}{$w_2$}
 \psfrag{r1}{$u_1$} \psfrag{r2}{$u_2$} \psfrag{t1}{$v_1$} \psfrag{t2}{$v_2$}
\centerline{\includegraphics[scale=0.6]{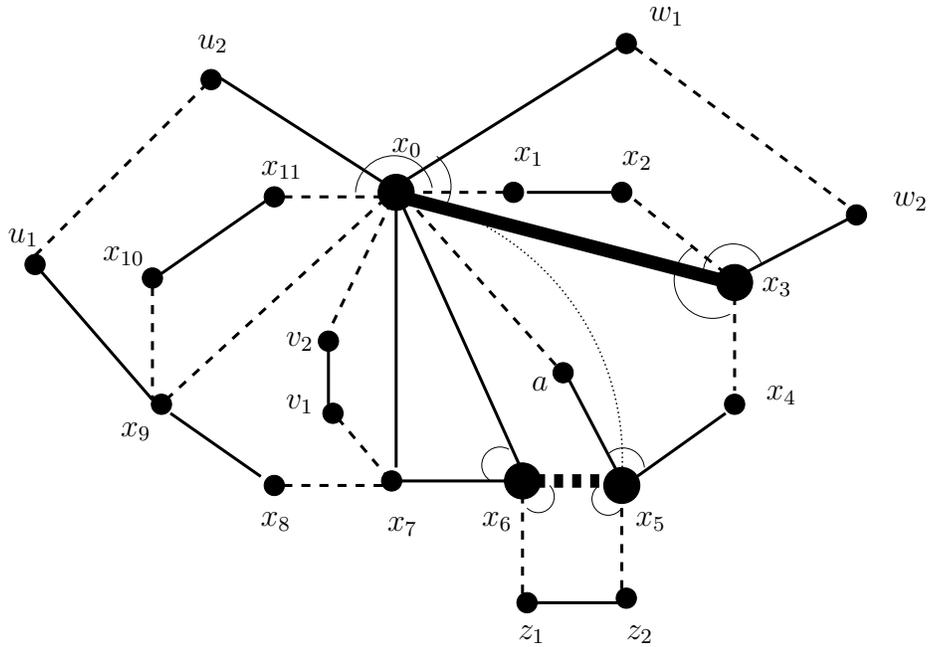}}
\caption{After Phase 1, with 2 interesting edges, 4 bad vertices, 8 bad pairs.}
\label{example3}
\end{figure}
We see two interesting edges (the shortcut edge and the odd chord), four bad vertices and
eight bad pairs.  

Next we launch into Phase 2, and after one step we see the situation of Figure~\ref{example4}.
\begin{figure}[ht]
 \psfrag{v}{$x_0$}\psfrag{x00}{$a$}\psfrag{y00}{$x_5$}
 \psfrag{y10}{$x_2$}\psfrag{x10}{$x_1$}\psfrag{x11}{$x_{11}$}
 \psfrag{y11}{$x_{10}$}\psfrag{y01}{$x_7$}\psfrag{x01}{$x_6$}
 \psfrag{u1}{$x_4$} \psfrag{u2}{$x_3$} \psfrag{u3}{$x_9$} \psfrag{u4}{$x_8$}
 \psfrag{p1}{$z_1$} \psfrag{p2}{$z_2$} \psfrag{w1}{$w_1$} \psfrag{w2}{$w_2$}
 \psfrag{r1}{$u_1$} \psfrag{r2}{$u_2$} \psfrag{t1}{$v_1$} \psfrag{t2}{$v_2$}
\centerline{\includegraphics[scale=0.6]{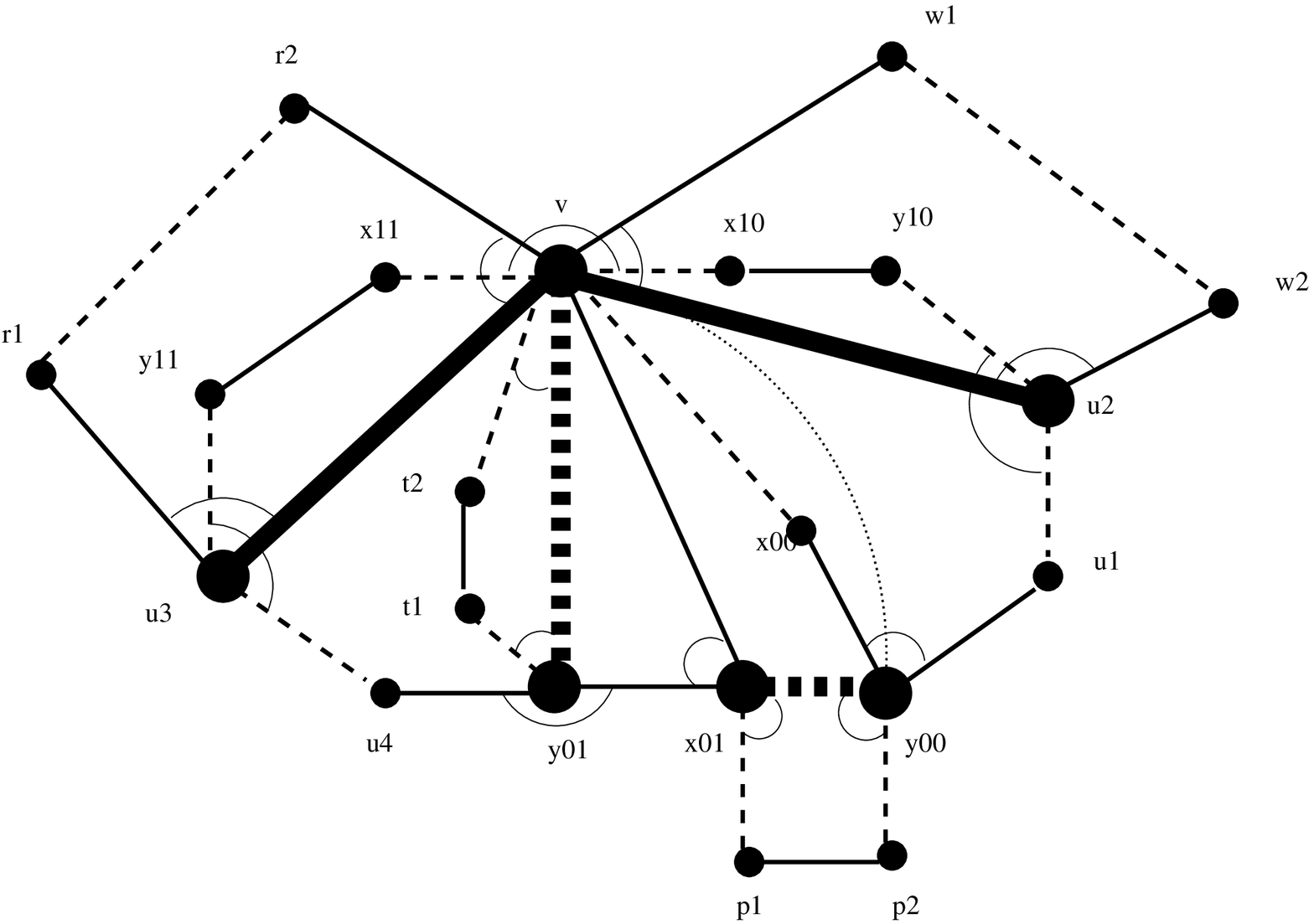}}
\caption{After one step of Phase 2, with 4 interesting edges, 6 bad vertices, 14 bad pairs.}
\label{example4}
\end{figure}
Now there are three odd chords as well as the shortcut edge, giving four interesting edges.
None of them are bad as we assumed that they belong to $H$.  But this means that
once they are disturbed, they each contribute two bad pairs.  Overall the six bad vertices
contribute 14 bad pairs.  At $x_0$ we see two bad pairs of each colour, and elsewhere we
see one bad pair of each colour.  

Take one more step of Phase 2 to obtain Figure~\ref{example5}.
\begin{figure}[ht]
 \psfrag{v}{$x_0$}\psfrag{x00}{$a$}\psfrag{y00}{$x_5$}
 \psfrag{y10}{$x_2$}\psfrag{x10}{$x_1$}\psfrag{x11}{$x_{11}$}
 \psfrag{y11}{$x_{10}$}\psfrag{y01}{$x_7$}\psfrag{x01}{$x_6$}
 \psfrag{u1}{$x_4$} \psfrag{u2}{$x_3$} \psfrag{u3}{$x_9$} \psfrag{u4}{$x_8$}
 \psfrag{p1}{$z_1$} \psfrag{p2}{$z_2$} \psfrag{w1}{$w_1$} \psfrag{w2}{$w_2$}
 \psfrag{r1}{$u_1$} \psfrag{r2}{$u_2$} \psfrag{t1}{$v_1$} \psfrag{t2}{$v_2$}
\centerline{\includegraphics[scale=0.6]{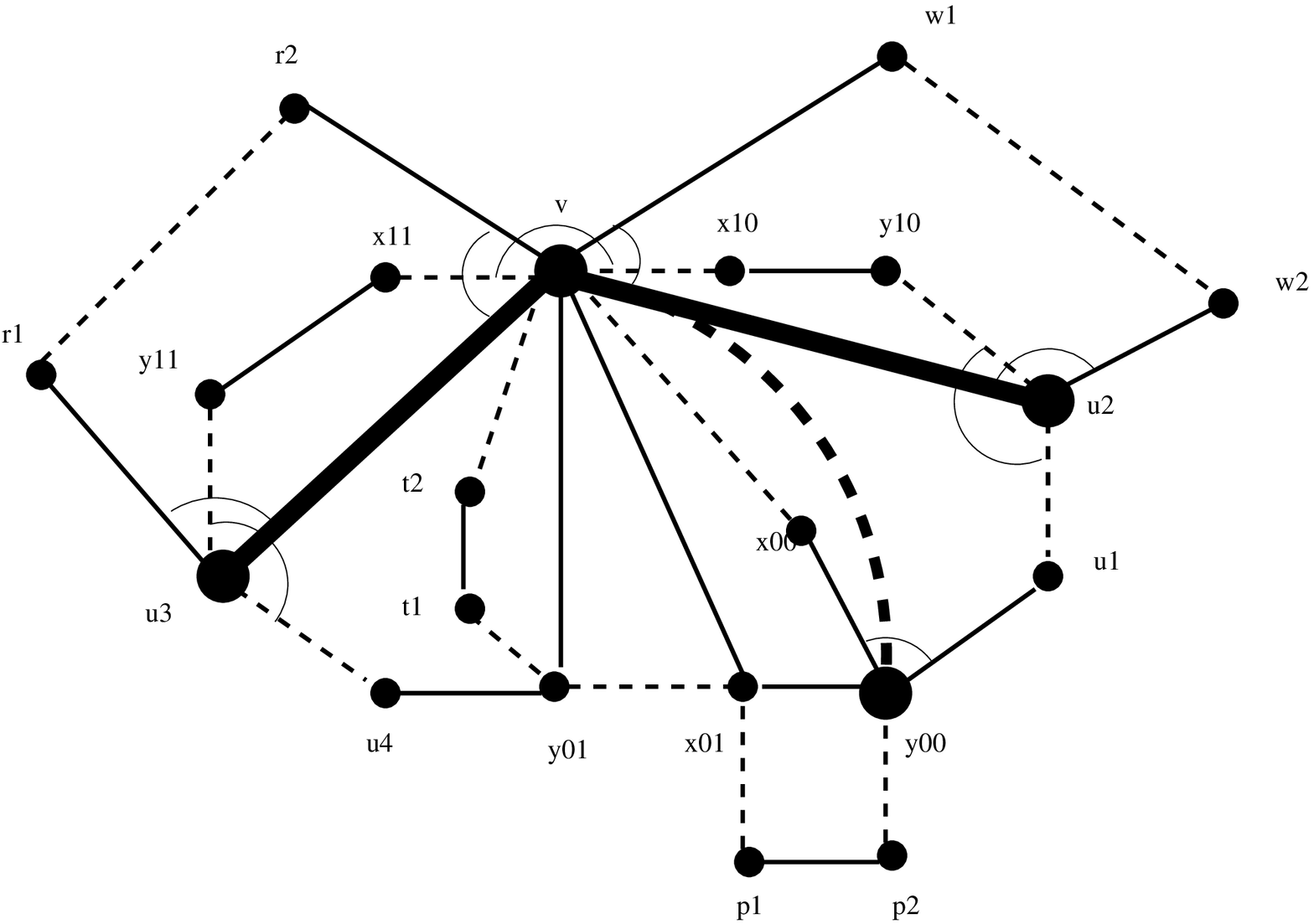}}
\caption{After two steps of Phase 2, with 3 interesting edges, 4 bad vertices, 8 bad pairs.}
\label{example5}
\end{figure}
The shortcut edge is now restored, so we just have the three odd chords.
Two of these belonged to the symmetric difference originally and hence contribute two
bad pairs each (one at each endpoint).  The other, $x_0 x_5$ was assumed to belong
to $G\cap G'$ and does not contribute any bad pairs.  (If it also belonged to the
symmetric difference then we would have 10 bad pairs at this step.)

Now we complete Phase 2 in one more step, see Figure~\ref{example6}.
\begin{figure}[ht]
 \psfrag{v}{$x_0$}\psfrag{x00}{$a$}\psfrag{y00}{$x_5$}
 \psfrag{y10}{$x_2$}\psfrag{x10}{$x_1$}\psfrag{x11}{$x_{11}$}
 \psfrag{y11}{$x_{10}$}\psfrag{y01}{$x_7$}\psfrag{x01}{$x_6$}
 \psfrag{u1}{$x_4$} \psfrag{u2}{$x_3$} \psfrag{u3}{$x_9$} \psfrag{u4}{$x_8$}
 \psfrag{p1}{$z_1$} \psfrag{p2}{$z_2$} \psfrag{w1}{$w_1$} \psfrag{w2}{$w_2$}
 \psfrag{r1}{$u_1$} \psfrag{r2}{$u_2$} \psfrag{t1}{$v_1$} \psfrag{t2}{$v_2$}
\centerline{\includegraphics[scale=0.6]{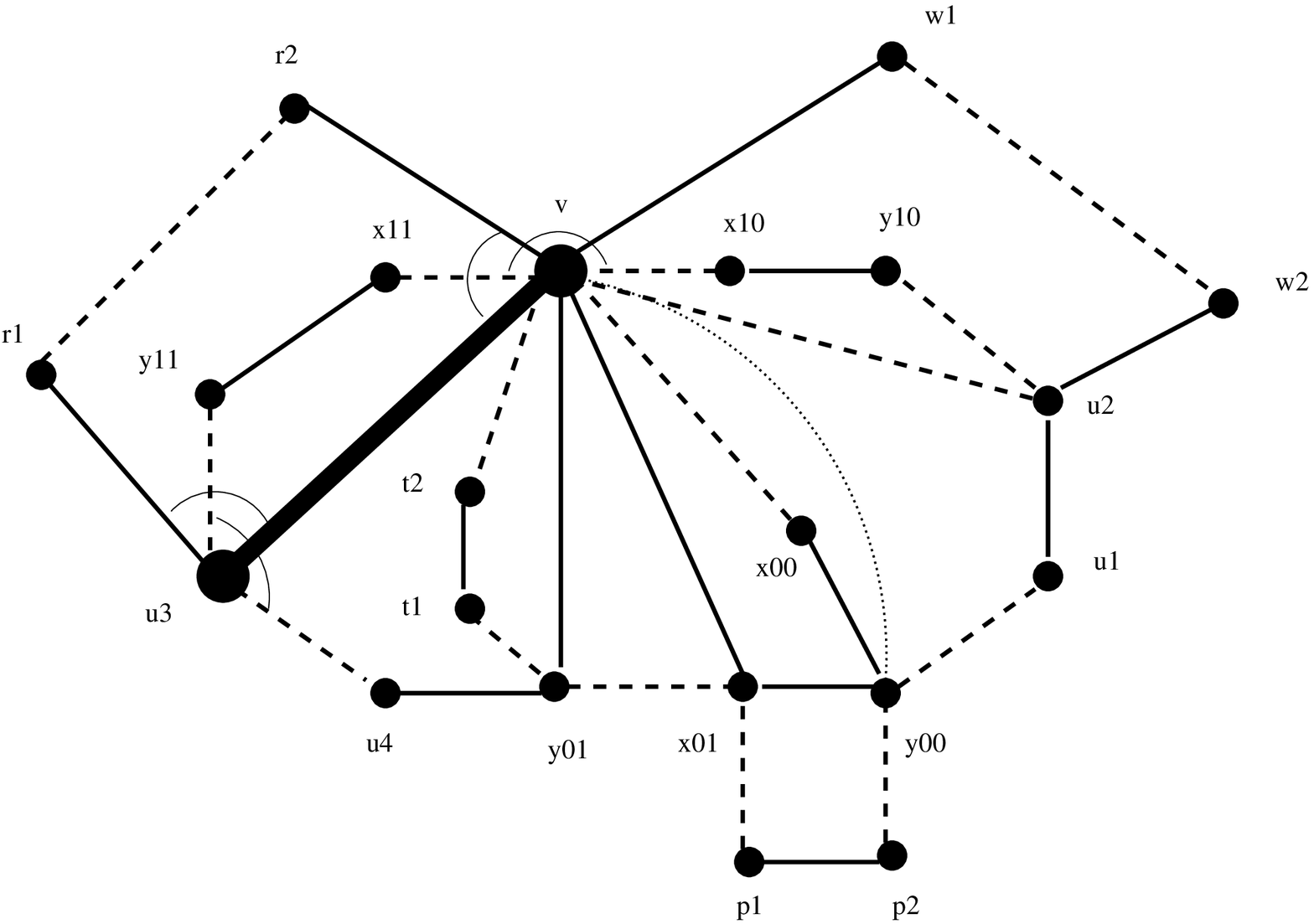}}
\caption{After Phase 2, with 1 interesting edge, 2 bad vertices, 4 bad pairs.}
\label{example6}
\end{figure}

We only have one interesting odd chord left, giving two bad vertices and four bad pairs.  
Then Phase 3 takes only one step and produces Figure~\ref{example7}, with the 2-circuit
processed and all interesting edges restored.
\begin{figure}[ht]
 \psfrag{v}{$x_0$}\psfrag{x00}{$a$}\psfrag{y00}{$x_5$}
 \psfrag{y10}{$x_2$}\psfrag{x10}{$x_1$}\psfrag{x11}{$x_{11}$}
 \psfrag{y11}{$x_{10}$}\psfrag{y01}{$x_7$}\psfrag{x01}{$x_6$}
 \psfrag{u1}{$x_4$} \psfrag{u2}{$x_3$} \psfrag{u3}{$x_9$} \psfrag{u4}{$x_8$}
 \psfrag{p1}{$z_1$} \psfrag{p2}{$z_2$} \psfrag{w1}{$w_1$} \psfrag{w2}{$w_2$}
 \psfrag{r1}{$u_1$} \psfrag{r2}{$u_2$} \psfrag{t1}{$v_1$} \psfrag{t2}{$v_2$}
\centerline{\includegraphics[scale=0.6]{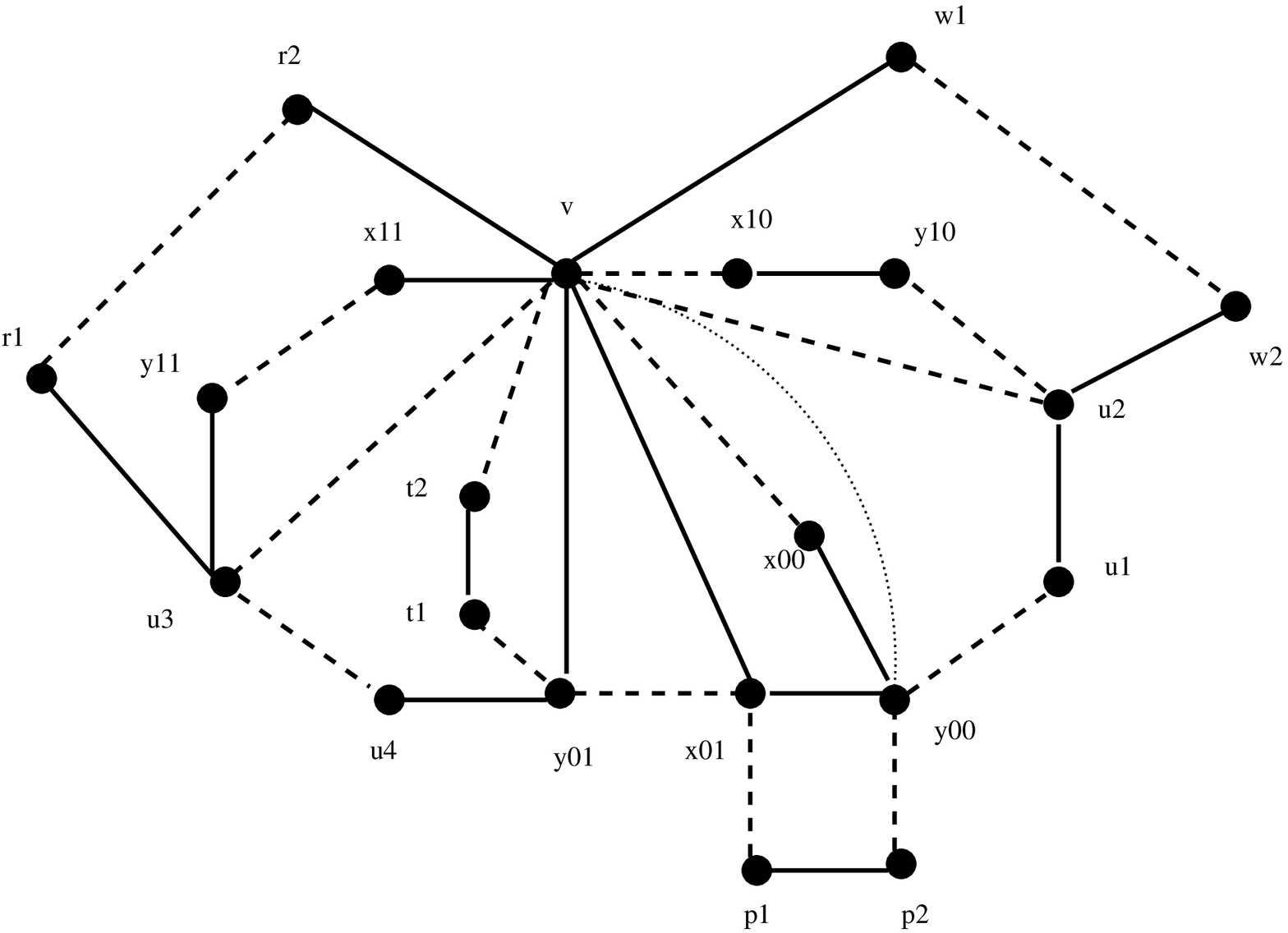}}
\caption{After Phase 3: processing of the 2-circuit is complete.}
\label{example7}
\end{figure}

Finally, there are four remaining circuits in $\mathcal{C}$, each consisting of a 4-cycle that
can be processed using a single switch.  Performing these four switches completes
the construction of the canonical path from $G$ to $G'$ with respect to the pairing $\psi$.
(No bad pairs are created during these four final switches.)

\end{document}